\pdfoutput=1

\documentclass[a4paper,oneside,reqno]{amsart}

\usepackage[english]{babel}
\usepackage[utf8]{inputenc}		
\usepackage[scaled]{helvet}		
\usepackage{courier}			
\usepackage{eulervm}			
\usepackage[bb=boondox,bbscaled=1.05,scr=dutchcal]{mathalfa}	%
\usepackage{dsfont}
\normalfont

\usepackage{indentfirst}				
\usepackage{tabularx,booktabs}			
\usepackage{caption,subcaption}			
\usepackage{fullpage}				    
\usepackage[english]{varioref}			
\usepackage[dvipsnames]{xcolor}			
\usepackage{hyperref}					
\usepackage{bookmark}					
\usepackage{textcomp}
\usepackage{rotating,afterpage,pdflscape}
\usepackage{enumerate}

\definecolor{webbrown}{rgb}{0.65, 0.16, 0.16}

\hypersetup{
colorlinks=true, linktocpage=true, pdfstartpage=1, pdfstartview=FitV,breaklinks=true, pdfpagemode=UseNone, pageanchor=true, pdfpagemode=UseOutlines,plainpages=false, bookmarksnumbered, bookmarksopen=true, bookmarksopenlevel=1,hypertexnames=true, pdfhighlight=/O,urlcolor=webbrown, linkcolor=RoyalBlue, citecolor=ForestGreen}

\usepackage{graphicx}		
\usepackage{float}
\usepackage{tikz}			
\usepackage{tikz-cd}		
\usetikzlibrary{
  arrows,decorations.pathreplacing,decorations.markings,shapes.geometric,positioning,calc,fadings,angles
	}
\tikzfading[name=inner fade, inner color=transparent!0, outer color=transparent!100]
\graphicspath{{Images/}}

\usepackage[textsize=footnotesize,textwidth=0.85in]{todonotes}
\presetkeys%
    {todonotes}%
    {color=Dandelion}{}%
\usepackage{amsmath,amssymb,amsthm,mathtools}
\usepackage{bm,braket,faktor,stmaryrd}

\numberwithin{equation}{section}

\newcommand{\ZZ}{\mathbb{Z}}
\newcommand{\QQ}{\mathbb{Q}}

\newcommand{\CC}{\mathbb{C}}

\newcommand{\ch}{{\rm ch}}
\newcommand{\Ch}{{\Omega}}

\newcommand{\stirlingfirst}[2]{\genfrac{[}{]}{0pt}{}{#1}{#2}}
\newcommand{\stirlingsecond}[2]{\genfrac{\{}{\}}{0pt}{}{#1}{#2}}

\DeclareMathOperator{\Aut}{Aut}
\newcommand{\rank}{{\rm rank}}

\newcommand{\MV}{\mathsf{MV}}

\theoremstyle{plain}
\newtheorem{thm}{Theorem}[section]

\newtheorem{prop}[thm]{Proposition}

\newtheorem{lem}[thm]{Lemma}
\newtheorem{cor}[thm]{Corollary}

\theoremstyle{definition}

\newtheorem{rem}[thm]{Remark}

\usepackage[citestyle=numeric, minbibnames=5, maxbibnames=99, giveninits=true, backend=biber]{biblatex}
\renewbibmacro{in:}{}
\usepackage{csquotes}
\title[Harer--Zagier via intersection theory]{An intersection-theoretic proof of the Harer--Zagier formula}

\author{Alessandro Giacchetto }
\address{A.~G.: 
	Universit\'e Paris-Saclay, CNRS, CEA, Institut de physique th\'eorique, 91191 Gif-sur-Yvette, France 
	\emph{\&} 
	Max Planck Institut fur Mathematik, Vivatsgasse 7, 53111 Bonn, Germany
}
\email{alessandro.giacchetto@ipht.fr}

\author{Danilo Lewa\'{n}ski }
\address{D.~L.: 
	Universit\'e de G\`{e}neve, Section de Math\'{e}matiques, 24, rue du G\'{e}n\'{e}ral Dufour, 1211 Geneva 4, Switzerland 
	\emph{\&} 
	Universit\'e Paris-Saclay, CNRS, CEA, Institut de physique th\'eorique, 91191 Gif-sur-Yvette, France 
	\emph{\&} 
	Institut des Hautes \'Etudes Scientifiques, le Bois-Marie, 35 route de Chartres, 91440 Bures-sur-Yvette, France 
}
\email{danilo.lewanski@ipht.fr}

\author{Paul Norbury }
\address{P.~N.: 
	School of Mathematics and Statistics, University of Melbourne, Australia 3010
}
\email{norbury@unimelb.edu.au}
	
\date{}


\begin{document}

\begin{abstract}
	We provide an intersection-theoretic formula for the Euler characteristic of the moduli space of smooth curves. 
	%
	This formula reads purely in terms of Hodge integrals and, as a corollary, the standard calculus of tautological classes gives a new short proof of the Harer--Zagier formula. 
	%
	Our result is based on the Gauss--Bonnet formula, and on the observation that a certain parametrisation of the $\Omega$-class \textemdash~the Chern class of the universal $r$-th root of the twisted log canonical bundle \textemdash~provides the Chern class of the log tangent bundle to the moduli space of smooth curves.
	%
	Being $\Omega$-classes by now employed in many enumerative problems, mostly recently found and at times surprisingly different from each other, we dedicate some work to produce an extensive list of their general properties: extending existing ones, finding new ones, and writing down some only known to the experts.
\end{abstract}

\maketitle

\section{Introduction and results}
\label{sec:intro}

Let $\mathcal{M}_{g,n}$ be the moduli space of smooth curves of genus $g$ with $n$ labeled and distinct marked points. Moduli spaces of curves are a topic of great interest both within the pure algebraic geometry and, arguably even more charmingly, in the relation between algebraic geometry and different branches of mathematics and physics: string theory, mirror symmetry, Gromov--Witten theory, random matrix models, integrable systems and integrable hierarchies, as well as recent methods such as topological recursion in the sense of Eynard and Orantin.

\medskip

The Euler characteristic $\chi_{g,n}$ of $\mathcal{M}_{g,n}$ represents one of the most fundamental invariants of these spaces, and is computed by the famous Harer--Zagier formula.

\begin{thm}[{\emph{Harer--Zagier formula}, \cite{HZ86}, 1986}] \label{thm:HZ}
	The orbifold Euler characteristic of $\mathcal{M}_{g,n}$ is given by
	\begin{equation}
 		\chi_{g,n} =
 		\begin{cases}
			(-1)^{n-3}(n - 3)! & g = 0, n \ge 3, \\
			(-1)^n \frac{(n - 1)!}{12} & g = 1, n \ge 1, \\
			(-1)^n (2g - 3 + n)! \frac{B_{2g}}{2g (2g - 2)!} & g \ge 2, n \ge 0,
		\end{cases}
	\end{equation}
	where $B_{2g}$ are Bernoulli numbers.
\end{thm}

This formula has been proved with quite different methods: via the combinatorics of polygon/graphs on surfaces and random matrix models, via representation theory, via topological recursion, via semi-infinite wedge formalism, see for instance \cite{Lew19,MPS21} and references therein.

\medskip

We provide the first expression of $\chi_{g,n}$ in terms of intersection theory of explicit tautological classes on the moduli spaces of stable curves, which is the main result of the paper.

\begin{thm}\label{thm:main}
	The Euler characteristic $\chi_{g,n}$ is given by the following Hodge integrals:
	\begin{equation}   \label{mainform}
		\chi_{g,n}
		=
		(-1)^{3g - 3 + n}
		\sum_{\ell \ge 0} \frac{1}{\ell!} \sum_{i=0}^g
			\int_{\overline{\mathcal{M}}_{g,n + \ell}} \frac{ \lambda_i}{ \prod_{j = 1}^{\ell} (1 + \psi_{n+j})} \psi_{n+1}^2 \cdots \psi_{n+\ell}^2.
	\end{equation}
\end{thm}
The product is $1$ when $\ell=0$, so the integrand consists only of $\lambda_i$.  Beginning with this formula, a short manipulation of linear Hodge integrals provides the Harer--Zagier formula.

\begin{cor}\label{cor:HZ}
	The Harer--Zagier formula holds true.
\end{cor}
The techniques used by Harer and Zagier in \cite{HZ86} required $n>0$ and deduced the $n=0$ case via an exact sequence.  The formula \eqref{mainform} applies equally for all $n\geq 0$.

\subsection{A parallel with Masur--Veech volumes via $\Omega$-classes}

In fact, Theorem~\ref{thm:main} relies on the observation that a specific parametrisation of the $\Omega$-classes (see Section \ref{sec:Omega:classes}) is the Chern class of the logarithmic tangent bundle. Thus, the following result is a consequence of Gauss--Bonnet.

\begin{prop}\label{prop:chi:Omega}
	The Euler characteristic $\chi_{g,n}$ is given by the following $\Omega$-integral:
	\begin{equation}\label{eqn:chi:Omega}
		\chi_{g,n}
		=
		\int_{\overline{\mathcal{M}}_{g,n}} \Ch(1,-1;0).
	\end{equation}
\end{prop}

An explicit formula for the $\Omega$-classes in terms of $\psi$-classes, $\kappa$-classes and boundary divisors was proved by Chiodo \cite{Chi08}; expanding the $\Omega$-class in this form provides the statement of Theorem~\ref{thm:main}. 

\medskip

We would like here to make a brief parallel with a slightly different story. Masur--Veech volumes $\pi^{-(6g - 6 + 2n)} \MV_{g,n} \in \QQ$ of the principal stratum of the moduli space of quadratic differentials have been shown to have a cohomological representation, given by the Segre class \textemdash~as opposed to the Chern class \textemdash~of the logarithmic tangent bundle (to make the comparison cleaner, we ignore the normalisation constant $2^{2g+1} (4g-4+n)!/(6g - 7 + 2n)!$ due to the labelling of simple poles and zeros of the quadratic differentials, and the normalisation of the Masur--Veech measure).

\begin{thm}[{\cite{CMS+19}}] \label{thm:CMS}
	The Masur--Veech volume $\MV_{g,n}$ is given by the following $\Omega$-integral:
	\begin{equation}\label{eqn:MVOmega}
		\frac{\MV_{g,n}}{\pi^{6g - 6 + 2n}}
		=
		(-1)^{3g-3+n}
		\int_{\overline{\mathcal{M}}_{g,n}} \Ch(1,2;0)
		=
		\int_{\overline{\mathcal{M}}_{g,n}} \bigl( \Ch(1,-1;0) \bigr)^{-1}.
	\end{equation}
\end{thm}

The integrand class is a Segre class instead of a Chern class, which produces a shift in the parameters of the $\Omega$-class (see \ref{rem:Chern:Segre} for a more general relation). In a similar fashion, expanding the Chern characters of the $\Omega$-class, one can express the Masur--Veech volumes $\MV_{g,n}$ as an explicit finite linear combination of Hodge integrals. In the same way Equation~\eqref{eqn:MVOmega} is a ``dual'' statement to Equation~\eqref{eqn:chi:Omega}, the following statement corresponds to our result for the Euler characteristic in Theorem~\ref{thm:main}.

\begin{thm}[{\cite{CMS+19}}]
	The Masur--Veech volume $\MV_{g,n}$ is given by the following  Hodge integrals:
	\begin{equation}
	    \frac{\MV_{g,n}}{\pi^{6g - 6 + 2n}}
	    =
		\sum_{\ell \ge 0} \frac{1}{\ell!} \sum_{i=0}^g
			\int_{\overline{\mathcal{M}}_{g,n + \ell}} \lambda_i \psi_{n+1}^2 \cdots \psi_{n+\ell}^2.
	\end{equation}
\end{thm}

In fact this result had tremendous consequences on the study of Masur--Veech volumes of quadratic differentials. First, the statement settled two conjectures on the large $n$ behaviour of the Masur--Veech volumes and their relative Siegel--Veech constants elaborated in \cite{ABCDGLW19}. Then, in two different works, Kazarian and Yang--Zagier produced fast recursive methods to compute these Masur--Veech volumes by exploiting KP-type (resp. ILW-type) integrability properties of such linear Hodge integrals. Moreover, Aggarwal \cite{Agg21} provided their large $g$ limit behaviour, originally conjectured by Aggarwal--Delecroix--Goujard--Zograf--Zorich.

\medskip

To conclude the parallelism, we point out that both enumerative problems $\MV_{g,n}$ and $\chi_{g,n}$ are generated by topological recursion in the sense of Eynard and Orantin \cite{ABCDGLW19,DN11}.

\medskip

Being $\Omega$-classes by now employed in around twenty enumerative problems, mostly recently found and at times surprisingly different from each other, we dedicate some work to produce an extensive list of their general properties: extending existing ones, finding new ones, and writing down some only known to the experts.

\subsection*{Acknowledgements}

This work is partly a result of the ERC-SyG project, Recursive and Exact New Quantum Theory (ReNewQuantum) which received funding from the European Research Council (ERC) under the European Union's Horizon 2020 research and innovation programme under grant agreement No 810573. A.~G.~has been supported by the Max-Planck-Gesellschaft and the Institut de Physique Th\`{e}orique Paris (IPhT), CEA, Universit\'{e} de Saclay. D.~L.~has been supported by the Section de Math\'{e}matique de l'Universit\'{e} de Gen\`{e}ve, by the Institut de Physique Th\`{e}orique Paris (IPhT), CEA, Universit\'{e} de Saclay, by the Institut des Hautes \'Etudes Scientifiques (IHES), Universit\'{e} de Saclay, and by the INdAM group GNSAGA. P.~N.~has been supported under the Australian Research Council Discovery Projects funding scheme project number DP180103891. Both A.~G. and D.~L. have been supported by the University of Melbourne who hosted the research visit which brought to this collaboration.

\smallskip

The authors thank G.~Borot, R.~Cavalieri and M.~Möller for useful discussions.

\subsection*{Outline of the paper}

In Section \ref{sec:Omega:classes} we provide the necessary background on $\Omega$-classes. In Section \ref{sec:proofs} we provide the proofs to the statements presented in the introduction. In Section \ref{sec:Omega:properties} we establish a list of properties of the $\Omega$-classes and prove them, as well as applying some of these properties to prove vanishing of integrals of $\Omega$-classes.

\section{Background on \texorpdfstring{$\Omega$}{Omega}-classes}
\label{sec:Omega:classes}

In \cite{Mum83}, Mumford derived a formula for the Chern character of the Hodge bundle on the moduli space of curves $\overline{\mathcal{M}}_{g,n}$ in terms of tautological classes and Bernoulli numbers. Among various applications, such class appears in the celebrated ELSV formula \cite{ELSV01}, named after its four authors Ekedahl, Lando, Shapiro, Vainshtein, that is an equality between simple Hurwitz number and an integral over the moduli space of stable curves.

\medskip

A generalisation of Mumford's formula was found by Chiodo in \cite{Chi08}. The moduli space $\overline{\mathcal{M}}_{g,n}$ is substituted by the proper moduli stack $\overline{\mathcal{M}}_{g;a}^{r,s}$ of $r$-th roots of the line bundle
\begin{equation}
	\omega_{\log}^{\otimes s}\biggl(-\sum_{i=1}^n a_i p_i \biggr),
\end{equation}
where $\omega_{\log} = \omega(\sum_i p_i)$ is the log-canonical bundle, $r$ and $s$ are integers with $r$ positive, and $a_1, \ldots, a_n$ are integers satisfying the modular constraint
\begin{equation}
	a_1 + a_2 + \cdots + a_n \equiv (2g-2+n)s \pmod{r}.
\end{equation}
This condition guarantees the existence of a line bundle whose $r$-th tensor power is isomorphic to $\omega_{\log}^{\otimes s}(-\sum_i a_i p_i)$. Let $\pi \colon \overline{\mathcal{C}}_{g;a}^{r,s} \to \overline{\mathcal{M}}_{g;a}^{r,s}$ be the universal curve, and $\mathcal{L} \to \overline{\mathcal C}_{g;a}^{r,s}$ the universal $r$-th root. The {\em type} of the marked point $p_i$ is $a_i \pmod{r}$. The lower strata of $\overline{\mathcal{M}}_{g;a}^{r,s}$ consist of moduli spaces $\overline{\mathcal{M}}_{g';a'}^{r,s}$ and in particular a node of a stable curve has type $a$ and $r-a$, again $\pmod{r}$, at the two branches.  In complete analogy with the case of moduli spaces of stable curves, one can define $\psi$-classes and $\kappa$-classes. There is moreover a natural forgetful morphism
\begin{equation}
	\epsilon \colon
	\overline{\mathcal{M}}^{r,s}_{g;a}
	\longrightarrow
	\overline{\mathcal{M}}_{g,n}
\end{equation}
which forgets the line bundle, otherwise known as the spin structure. It can be turned into an unramified covering (in the orbifold sense) of degree $r^{2g - 1}$ by slightly modifying the structure of $\overline{\mathcal{M}}_{g,n}$, introducing an extra $\ZZ/r\ZZ$ stabiliser for each node of each stable curve (see \cite{JPPZ17}).

\medskip

Let $B_m(x)$ denote the $m$-th Bernoulli polynomial, that is the polynomial defined by the generating series
\begin{equation}
	\frac{te^{tx}}{e^t - 1} = \sum_{m = 0}^{\infty} B_{m}(x)\frac{t^m}{m!}.
\end{equation}
The evaluations $B_m(0) = (-1)^m B_m(1) = B_m$ recover the usual Bernoulli numbers. Chiodo's formula provides an explicit formula for the Chern characters of the derived pushforward of the universal $r$-th root $\ch_m(r,s;a) = \ch_m(R^{\bullet} \pi_{\ast}{\mathcal L})$.

\begin{thm}[\cite{Chi08}]
	The Chern characters $\ch_m(r,s;a)$ of the derived pushforward of the universal $r$-th root have the following explicit expression in terms of $\psi$-classes, $\kappa$-classes, and boundary divisors:
	\begin{equation} \label{eqn:Chiodo:formula}
		\ch_m(r,s;a)
		=
		\frac{B_{m+1}(\tfrac{s}{r})}{(m+1)!} \kappa_m
		-
		\sum_{i=1}^n \frac{B_{m+1}(\tfrac{a_i}{r})}{(m+1)!} \psi_i^m
		+
		\frac{r}{2} \sum_{a=0}^{r-1} \frac{B_{m+1}(\tfrac{a}{r})}{(m+1)!} \, j_{a,\ast} \frac{(\psi')^m - (-\psi'')^m}{\psi' + \psi''}. 
	\end{equation}
	Here $j_a$ is the boundary morphism that represents the boundary divisor such that the two branches of the corresponding node are of type $a$ and $r-a$, and $\psi',\psi''$ are the $\psi$-classes at the two branches of the node.
\end{thm}

We can then consider the pushforward to the moduli space of stable curves of the family of Chern classes
\begin{equation}\label{eqn:Omega}
	\Ch_{g,n}^{[x]}(r,s;a)
	=
	\epsilon_{\ast}
	\exp{\Biggl(
		\sum_{m=1}^\infty (-x)^m (m-1)! \, \ch_m(r,s;a)
	\Biggr)}
	\in
	H^{\textup{even}}(\overline{\mathcal{M}}_{g,n}).
\end{equation}
We will omit the variable $x$ when $x = 1$ and the indices $(g,n)$ whenever it is clear from the context. Notice that we recover Mumford's formula for the Hodge class when $r = s = 1$ and $a = (1,\dots,1)$. For $r = 1$, general $s$ and $a = (s,\dots,s)$, we get the generalised Hodge classes considered by Bini in \cite{Bin03}. For any $r \in \ZZ_+$ and $s \in \ZZ$, we refer to these classes as $\Omega$-classes, or $\Omega$-CohFT when referring to them as a collection or when exploiting some of their features as a Cohomological Field Theory (CohFT), based on the following statement.

\begin{thm}[{\cite{LPSZ17}, $\Omega$-classes form a CohFT}]
	Let $r$ be a positive integer and let $V = \langle v_1, \dots, v_r \rangle_{\mathbb{C}}$ be a vector space. For any $s \in \ZZ$, the collection of maps 
	\begin{equation}
		\Omega_{g,n}(r,s; \bullet): V^{\otimes n} \longrightarrow H^{\textup{even}}(\overline{\mathcal{M}}_{g,n})
	\end{equation}
	associating to a vector $(v_{a_1}, \dots, v_{a_n})$ the cohomology class $\Omega_{g,n}(r,s; a_1, \dots, a_n)$, and extended by multilinearity, forms a cohomological field theory with associated $V$-metric $\eta$ defined by 
	\begin{equation}
		\eta(v_a, v_b) = \frac{1}{r} \delta_{a+b \equiv 0 \pmod{r}}		
	\end{equation}
	and whose unit is flat whenever $s$ belongs to the range $[0,r]$.
\end{thm}

Notice that the above result is restricted to $1 \le a_1,\dots,\le r$. A relation among $\Omega$-classes with $a_i$ possibly outside this range will be given in Section \ref{sec:Omega:properties}.

\subsection{$\Omega$-classes as a sum over stable graphs}

By expanding the exponential \eqref{eqn:Chiodo:formula}, we find an expression of the $\Omega$-classes as a sum over decorated stable graphs in which vertices, leaves, and edges carry cohomology classes multiplied by sums of products of Bernoulli polynomials. However, a correct expansion of \eqref{eqn:Chiodo:formula} into an expression in terms of stable graphs has to carefully take into account all possible self-intersections, and re-expand each self-intersected edge into the Chern class of its normal bundle. The result of this procedure is written down in clean form in the following statement.

\begin{cor}[{\cite[Corollary 4]{JPPZ17}}] \label{cor:Omega:Exp}
	The class $\Omega_{g,n}^{[x]}(r,s;a)$ is equal to
	\begin{equation}
	\begin{split}
		\sum_{\Gamma \in \mathsf{G}_{g,n}} \sum_{w \in \mathsf{W}_{\Gamma,r,s}}
		\frac{r^{2g-1-h^1(\Gamma)}}{|\Aut{(\Gamma)}|} \xi_{\Gamma,\ast}
		& \;
		\prod_{\mathclap{v \in V(\Gamma)}} \;
			\exp\left(
				\sum_{m \ge 1} \frac{(-x)^{m}B_{m+1}(\frac{s}{r})}{m(m+1)} \kappa_m(v)
			\right) \\
		\times \; & \;
		\prod_{\mathclap{\substack{e \in E(\Gamma) \\ e = (h,h')}}} \;
			\frac{1 - \exp\left(
				- \sum_{m \ge 1} \frac{(-x)^m B_{m+1}(\frac{w(h)}{r})}{m(m+1)} \bigl( (\psi_{h})^m - (-\psi_{h'})^m \bigr)
			\right)}{\psi_h + \psi_{h'}} \\
		\times \; & \;
		\prod_{i=1}^n \;
			\exp\left(
				- \sum_{m \ge 1} \frac{(-x)^m B_{m+1}(\frac{a_i}{r})}{m(m+1)} \psi_{i}^m
			\right).
	\end{split}
	\end{equation}
	Here $\mathsf{G}_{g,n}$ is the finite set of stable graphs of genus $g$ with $n$ legs, $\mathsf{W}_{\Gamma,r,s}$ is the finite set of half-edges decorations with an integer in $\{0, \dots, r-1\}$ in such a way that the leaf $i$ is decorated by $a_i$, decorations of half-edges forming the same edge $e \in E(\Gamma)$ sum up to $r$, and locally on each vertex $v \in V(\Gamma)$ the sum of all decorations is congruent to $(2g - 2 + n)s$ modulo $r$.
\end{cor}

\subsection{Riemann--Roch for $\Omega$-classes}
\label{subsec:RR}

The Riemann--Roch theorem for an $r$-th root $L$ of $\omega_{\log}^{\otimes s}(-\sum_i a_i p_i)$ provides the following relation:
\begin{equation}
	\frac{(2g - 2 + n)s - \sum_i a_i}{r} - g + 1 = h^0(C,L) - h^1(C,L).
\end{equation}
In some cases, i.e. for particular parametrisations of $r,s,a_i$ and for topologies $(g,n)$, it can happen that either $h^0$ or $h^1$ vanish, turning the derived pushforward $R^{\bullet} \pi_{\ast}{\mathcal L}$ into a vector bundle. If that happens, the Riemann--Roch formula provides a bound for the complex cohomological degree of $\Omega$:
\begin{equation}
	[\deg_{\CC} = k].\Omega_{g,n}(r,s;a) = 0, \qquad \qquad \text{ for } k > \rank(R^{\bullet} \pi_{\ast}{\mathcal L}),
\end{equation}
which are usually trivial or not depending on whether the rank $< 3g - 3 + n$. 
\begin{enumerate}
	\item
	One of these instances is provided in genus zero by the following result of Jarvis, Kimura, and Vaintrob.

	\begin{thm}[{\cite[Proposition 4.4]{JKV01}}] \label{thm:JKV}
		Let $g=0$, $n \geq 3$, $s=0$, and consider $a_i$ all strictly positive except for at most a single $a_j$ which can be positive, or zero, or equal to $-1$. Then every $r$-th root of $\omega_{\log}^{\otimes s}(-\sum_i a_i p_i)$ does not have global sections, that is, we have $h^0 = 0$.
	\end{thm}

	\noindent
	Under the condition of the theorem above, the rank of $R^{\bullet} \pi_{\ast}{\mathcal L}$ equals $h^1$, and therefore one gets:
	\begin{equation}\label{eqn:RR:JKV}
		[\deg_{\CC} = k].\Omega_{0,n}(r,s;a) = 0, \qquad \qquad \text{ for } k > \frac{\sum_i a_i }{r} - 1.
	\end{equation}

	\item
	Another instance is provided by the negative $s$ case with $a_i$ all positive. Let $s = -s'$ for $s'$ a positive integer. In this case $\omega_{C,\log}^{\otimes -s'}(-\sum a_i)$ has strictly negative degree on all connected components of every stratum, therefore implying again that $h^0 = 0$. Thus, the rank equals $h^1$ and one gets:
	\begin{equation}\label{eqn:RR:Norb}
		[\deg_{\CC} = k].\Omega_{g,n}(r,s;a) = 0,
		\qquad \qquad
		\text{for } k > \frac{(2g - 2 + n)s' + r(g-1) + \sum_i a_i }{r},
	\end{equation}
	which is interesting when $(2g - 2 + n)s' + r(g-1) + \sum_i a_i  < (3g - 3 +n)r$. For instance, if $r = 2$ and all $a_i = s' = 1$, one gets that $\Omega_{g,n}(2,-1;1^n)$ has top complex degree equal to $2g - 2 + n$, and that top degree in fact defines (up to prefactors) a cohomological field theory $\Theta_{g,n}$ with beautiful properties \cite{Nor21} (its intersection numbers are generated by topological recursion in the sense of Eynard and Orantin, the associated partition function is conjecturally a solution of the KdV hierarchy which arises from the BGW random matrix model, and it is related to the volumes of the moduli spaces of super Riemann surfaces). 
\end{enumerate}

\section{Proofs}
\label{sec:proofs}

In this section we provide the proofs to the statements in the introduction.

\begin{proof}[{Proof of Theorem \ref{thm:main}}]
	Firstly, let us recall a generalised Gauss--Bonnet formula, expressing the orbifold Euler characteristic of certain open orbifolds as integrals of the Chern class of the logarithmic cotangent bundle. A proof of the formula can be found in \cite{CMZ20}.

	\begin{prop}[Gauss--Bonnet for open orbifolds]
		Let $\overline{M}$ be a compact smooth $k$-dimensional orbifold, let $D$ be a normal crossing divisor and $\overline{M} = M \setminus D$. Then the orbifold Euler characteristic of $M$ can be computed as
		\begin{equation}
			\chi(M)
			=
			(-1)^k \int_{\overline{M}} c_k\bigl( \Omega_{\overline{M}}^1(\log{D}) \bigr),
		\end{equation}
		where $c_k(\Omega_{\overline{M}}^1(\log{D}))$ is the $k$-th Chern class of the logarithmic cotangent bundle.
	\end{prop}

	Let us apply the above proposition to compute the Euler characteristic $\chi_{g,n} = \chi(\mathcal{M}_{g,n})$. The fibre of the logarithmic cotangent bundle of $\overline{\mathcal{M}}_{g,n}$ over a curve $(C,p_1,\dots,p_n)$ is given by the space of quadratic differentials on $C$ with simple poles at the marked points, that is
	\begin{equation*}
		H^{0}\bigl(C, \omega_C^{\otimes 2}({\textstyle\sum_{i}} p_i) \bigr).
	\end{equation*}
	On the other hand, consider the $\Omega$-class with parameters $r = 1$, $s = -1$ and all $a_i = 0$. As explained in Subsection \ref{subsec:RR}, $\Omega$ is the Chern class of an actual bundle whose fibre over a curve $(C,p_1,\dots,p_n)$ is isomorphic to $H^{1}(C,(\omega_{C,\log})^{-1})$. By Serre duality
	\begin{equation*}
		H^{1}(C,(\omega_{C,\log})^{-1}) \cong H^{0}\bigl(C, \omega_C^{\otimes 2}({\textstyle\sum_{i}} p_i) \bigr)^{\vee}.	
	\end{equation*}
	Thus, we find Proposition \ref{prop:chi:Omega}:
	\begin{equation*}
		\chi_{g,n}
		=
		\int_{\overline{\mathcal{M}}_{g,n}} \Ch_{g,n}(1,-1;0).
	\end{equation*}
	Specialising formula~\eqref{eqn:Chiodo:formula}, we get $\Ch_{g,n}(1,-1;0) = \Lambda(-1) \exp( - \sum_{m \ge 1} \frac{1}{m} \kappa_m )$. This is a simple consequence of the identity $B_{m}(-1) = B_m + (-1)^m m$, together with Mumford's formula for $\lambda$-classes. Here $\Lambda(-1) = \sum_{i=1}^g (-1)^i \lambda_i$ is the total Chern class of the dual of the Hodge bundle. We can convert the evaluation of the above class into a combination of simple Hodge integrals, using Lemma~\ref{lem:push:kappa}: the values $v_k = -1$ are given by $u_m = -\frac{1}{m}$ through Equation~\eqref{eqn:change:u:v}. In turn, we find 
	\begin{equation*}
	\begin{split}
		\chi_{g,n}
		& =
		\int_{\overline{\mathcal{M}}_{g,n}}
			\Lambda(-1)
			\exp\biggl( - \sum_{m \ge 1} \frac{1}{m} \kappa_m \biggr) \\
		& =
		\int_{\overline{\mathcal{M}}_{g,n}} \Lambda(-1)
		+
		\sum_{\ell \ge 1} \frac{(-1)^{\ell}}{\ell!} \sum_{\mu_1,\dots,\mu_{\ell} \ge 1}
			\int_{\overline{\mathcal{M}}_{g,n + \ell}} \Lambda(-1)\prod_{j = 1}^{\ell} \psi_{n+j}^{\mu_j + 1}.
	\end{split}
	\end{equation*}
	Notice that the sum over $\ell$ terminates at $\ell = 3g - 3 + n$, and the sum over $\mu$'s is also finite since we have $\mu_1 + \cdots + \mu_{\ell} \le 3g - 3 + n$. We also observe that the first summand vanishes for degree reasons, unless $(g,n) = (0,3)$ or $(1,1)$. In these cases,
	\begin{equation*}
		\int_{\overline{\mathcal{M}}_{0,3}} \Lambda(-1)
		=
		\int_{\overline{\mathcal{M}}_{0,3}} 1 = 1,
		\qquad\qquad
		\int_{\overline{\mathcal{M}}_{1,1}} \Lambda(-1)
		=
		- \int_{\overline{\mathcal{M}}_{1,1}} \lambda_1 = - \frac{1}{24}.
	\end{equation*}
	Collapsing geometric series into their compact form allows a re-arranging of signs lead to the statement:
	\begin{equation*}
	\begin{split}
		\chi_{g,n}
		& =
		\sum_{\ell \ge 0} \frac{(-1)^{\ell}}{\ell!} \int_{\overline{\mathcal{M}}_{g,n + \ell}}
			\frac{\Lambda(-1)}{\prod_{j = 1}^{\ell} 1 - \psi_{n+j}} \psi_{n+1}^{2} \cdots \psi_{n+\ell}^{2} \\
		& =
		(-1)^{3g-3+n} \sum_{\ell \ge 0} \frac{1}{\ell!} \int_{\overline{\mathcal{M}}_{g,n + \ell}}
			\frac{\Lambda(1)}{\prod_{j = 1}^{\ell} 1 + \psi_{n+j}} \psi_{n+1}^{2} \cdots \psi_{n+\ell}^{2}.
	\end{split}
	\end{equation*}
	This concludes the proof of Theorem \ref{thm:main}.
\end{proof}

\begin{proof}[{Proof of Corollary~\ref{cor:HZ}}]
	Thanks to the intersection-theoretic expression of the Euler characteristic, together with an explicit formula for Hodge integrals due to Dubrovin--Yang--Zagier (see \cite[Section~1.3]{DYZ17}), we are able to give a new proof of the Harer--Zagier formula. 
	
	\medskip

	\textbf{Claim 1.} The Euler characteristic satisfies\footnote{The relation $\chi_{g,n+1} = - (2g - 2 + n) \chi_{g,n}$ easily follows from a short exact sequence involving mapping class groups (see \cite[Section~6]{HZ86}). Here we provide another proof that only uses the intersection-theoretic expression for $\chi_{g,n}$.} $\chi_{g,n+1} = -(2g-2+n) \chi_{g,n}$. Indeed, denoting the forgetful morphism by $\pi \colon \overline{\mathcal{M}}_{g,n+1} \to \overline{\mathcal{M}}_{g,n}$, we have
	\begin{equation*}
	\begin{split}
		\chi_{g,n + 1}& =
		\int_{\overline{\mathcal{M}}_{g,n+1}}
			\Lambda(-1)
			\exp\biggl( - \sum_{m \ge 1} \frac{1}{m} \kappa_m \biggr) \\
			& =
		\int_{\overline{\mathcal{M}}_{g,n+1}}
			\pi^*\Lambda(-1)
			\exp\biggl( - \sum_{m \ge 1} \frac{1}{m} (\pi^*\kappa_m +\psi_{n+1}^m)\biggr) \\
			& =
		\int_{\overline{\mathcal{M}}_{g,n+1}}
			\exp\biggl( - \sum_{m \ge 1} \frac{1}{m}\psi_{n+1}^m\biggr)\pi^*\Lambda(-1)
			\pi^*\exp\biggl( - \sum_{m \ge 1} \frac{1}{m} \kappa_m \biggr) \\
			& =
		\int_{\overline{\mathcal{M}}_{g,n+1}}
			(1-\psi_{n+1})\pi^*\Lambda(-1)
			\pi^*\exp\biggl( - \sum_{m \ge 1} \frac{1}{m} \kappa_m \biggr) \\
			& =
		-\int_{\overline{\mathcal{M}}_{g,n}}
			(\pi_*\psi_{n+1})\Lambda(-1)
			\exp\biggl( - \sum_{m \ge 1} \frac{1}{m} \kappa_m \biggr) \\
			& =
		- (2g-2+n) \int_{\overline{\mathcal{M}}_{g,n}}
			\Lambda(-1)
			\exp\biggl( - \sum_{m \ge 1} \frac{1}{m} \kappa_m \biggr).
	\end{split}
	\end{equation*}
	Here we used the fact $\pi^{\ast}$ is a ring homomorphism, together with the property for $\pi^{\ast}\Lambda(-1) = \Lambda(-1)$ and the relation $\pi^{\ast}\kappa_m = \kappa_m - \psi_{n+1}^m$.
	
	\medskip

	\textbf{Claim 2.} The Harer--Zagier relation holds true. Indeed, as a consequence of Claim 1, we just have to compute $\chi_{0,3}$, $\chi_{1,1}$ and $\chi_{g,0}$ for $g \ge 2$. Clearly, $\chi_{0,3} = 1$, so that $\chi_{0,n} = (-1)^{n-3} (n - 3)!$.	In genus one we compute
	\begin{equation*}
		\chi_{1,1}
		=
		\int_{\overline{\mathcal{M}}_{1,1}} \Lambda(-1)
		-
		\int_{\overline{\mathcal{M}}_{1,2}} \Lambda(-1) \psi_2^2
		=
		-\frac{1}{12}.
	\end{equation*}
	Thus, the relation $\chi_{1,n} = (-1)^{n} \frac{(n - 1)!}{12}$. Finally, in genus $g \ge 2$, we can use the explicit formula of \cite[Section~1.3]{DYZ17}, namely
	\begin{equation*}	
		\sum_{\ell \ge 1} \frac{1}{\ell!} \sum_{\mu_1,\ldots,\mu_\ell \ge 1}
			\int_{\overline{\mathcal{M}}_{g,\ell}} \Lambda(-1) \prod_{i=1}^{\ell} \psi_i^{\mu_i + 1}
		=
		\frac{B_{2g}}{2g(2g-2)}.
	\end{equation*}
	As the left-hand side equals $\chi_{g,0}$ by Proposition~\ref{prop:chi:Omega}, we have the thesis. This concludes the proof of Theorem~\ref{thm:main}.
\end{proof}

\section{Properties, symmetries and parameters shift of the \texorpdfstring{$\Omega$}{Omega}-CohFT}
\label{sec:Omega:properties}

Mainly within the past five years, applications of $\Omega$-classes in enumerative geometry have been blooming in the literature. Interestingly enough, these applications arise from quite different contexts and with different motivations.

\medskip

A complete list of recent papers employing $\Omega$-classes is out of the scope of this work and, as far as we know would likely be outdated soon. Instead, the aim of this section is to investigate, prove, extend and collect properties of $\Omega$-classes as a reference tool for interested mathematicians in the field.

\begin{thm}\label{thm:properties}
	Fix $g,n \geq 0$ integers such that $2g - 2 + n > 0$. Let $r$ and $s$ be integers with $r$ positive, and $1 \le a_1, \ldots, a_n \le r$ integers satisfying the modular constraint $a_1+a_2+\cdots+a_n \equiv (2g-2+n)s \pmod{r}$. The $\Omega$-classes satisfy the following properties.
	\begin{enumerate}[(i)]
		\item\label{prop:shift:s}
		Shift of $s$:
		\begin{equation}
			\Ch^{[x]}(r,s+r;a_1, \dots, a_n)
			=
			\Ch^{[x]}(r,s;a_1, \dots, a_n) \cdot \exp\Biggl( \sum_{m \ge 1} \frac{(-x)^m}{m} \left(\frac{s}{r}\right)^m \kappa_m \Biggr)
		\end{equation}

		\item\label{prop:shift:ai}
		Shift of $a_i$:
		\begin{equation}
			\Ch^{[x]}(r,s;a_1, \dots, a_i + r, \dots, a_n)
			=
			\Ch^{[x]}(r,s;a_1, \dots, a_n) \cdot\left( 1 + x\frac{a_i}{r}\psi_i\right)
		\end{equation}

		\item\label{prop:0:r:sym}
		Zero and $r$ symmetry:
		\begin{equation}
		\begin{aligned}
			\Ch(r,0;a_1, \dots, a_n)
			& =
			\Ch(r,r;a_1, \dots, a_n) \\
			\Ch(r,s; a_1, \dots, 0, \dots, a_n)
			& =
			\Ch(r,s;a_1, \dots, r, \dots, a_n)
		\end{aligned}
		\end{equation}

		\item\label{prop:pullback}
		Pullback property:
		\begin{equation}
			\Ch(r,s;a_1, \dots, a_n, s)
			=
			\pi^{\ast}\Ch(r,s;a_1, \dots, a_n)
		\end{equation}

		\item\label{prop:string}
		\textup{(String equation).} For formal variables $x_1, \dots, x_{n+1}$ we have:
		\begin{equation}\label{eqn:string:Omega}
			\int_{\overline{\mathcal{M}}_{g, n+1}}
				\!\!\!\!\!
				\frac{\Omega(r,s; a_1, \dots, a_{n}, a_{n+1} = s)}{\prod_{i=1}^{n+1} (1 - x_i \psi_i)} \Bigg{|}_{x_{n+1} = 0}
				=
				(x_1 + \dots + x_n) \int_{\overline{\mathcal{M}}_{g, n}} \frac{\Omega(r,s; a_1, \dots, a_{n})}{\prod_{i=1}^n (1 - x_i \psi_i)}
		\end{equation}

		\item\label{prop:dilaton}
		\textup{(Dilaton equation).} For formal variables $x_1, \dots, x_{n+1}$ we have:
		\begin{equation}\label{eqn:dilaton:Omega}
			\frac{\partial}{\partial x_{n+1}}
			\int_{\overline{\mathcal{M}}_{g, n+1}}
				\!\!\!\!\!
				\frac{\Omega(r,s; a_1, \dots, a_{n}, a_{n+1} = s)}{\prod_{i=1}^{n+1} (1 - x_i \psi_i)} \Bigg{|}_{x_{n+1} = 0}
			=
			(2g - 2 + n) \int_{\overline{\mathcal{M}}_{g, n}} \frac{\Omega(r,s; a_1, \dots, a_{n})}{\prod_{i=1}^n (1 - x_i \psi_i)}
		\end{equation}
	\end{enumerate}
	Iterating the first two properties above, one finds:
	\begin{enumerate}[(I)]
		\item\label{prop:mult:shifts:s}
		Multiple shifts of $s$:
		\begin{equation}
		\begin{split}
			& \Ch^{[x]}(r,s+Nr; a_1, \dots, a_n) \\
			& \qquad\quad
			=
			\Ch^{[x]}(r,s;a_1, \dots, a_n) \cdot \exp\Biggl( \sum_{m \ge 1} \frac{(-x)^m}{m} p_m\left(\frac{s}{r}, \dots, \frac{s}{r} + N - 1 \right) \kappa_m \Biggr)
		\end{split}
		\end{equation}
		where $p_m$ is the power sum symmetric polynomial of degree $m$.

		\item\label{prop:mult:shifts:ai}
		Multiple shifts of $a_i$:
		\begin{equation}
			\Ch^{[x]}(r,s; a_1, \dots, a_i + N r, \dots, a_n)
			=
			\Ch^{[x]}(r,s; a_1, \dots, a_n) \cdot \prod_{t=0}^{N-1} \left( 1 + x\left(\frac{a_i}{r} + t \right) \psi_i\right)
		\end{equation}
	\end{enumerate}
\end{thm}

\begin{rem}
	Property \eqref{prop:shift:s} was employed in \cite[Appendix~A]{CMS+19} for the case $r=1$, we extend it for general $r$. Property \eqref{prop:shift:ai} was observed by A.~Chiodo via geometric arguments for $s<0$, we extend it to general $s$. Property \eqref{prop:0:r:sym} is obvious, we list it here for completeness. Property \eqref{prop:pullback} was proved in \cite{LPSZ17} for $0 \leq s \leq r$, we extend it for arbitrary integer $s$, and it constitutes the strongest statement of the list. Properties \eqref{prop:string} and \eqref{prop:dilaton} have been proved in \cite{DLN16} for $s=0$ via topological recursion techniques, we extend it here to an arbitrary integer $s$. The other properties are new, to the best of our knowledge.
\end{rem}

\begin{rem}\label{rem:Chern:Segre}
	Another interesting property, which only holds for $r=1$, is a relation between two different parametrisations of $\Omega$-classes, which we refer to as \emph{Segre and Chern}:\footnote{The relation that one might expect from Serre duality applied to an $r$-th root of $\omega_{\log}^{\otimes s}(-\sum_i a_i p_i)$, i.e.
		\begin{equation*}
			\Ch^{[-x]}(r,r-s;r-a_1, \dots, r-a_n)
			=
			\left(\Ch^{[x]}(r,s;a_1, \dots, a_n)\right)^{-1},
		\end{equation*}
		is in fact false. As an explicit counterexample, in $(g,n) = (1,2)$ we have $\Omega^{[x]}(2,1;0,2)\Omega^{[-x]}(2,1;2,0) = 1 - \frac{3}{4}x^2\kappa_2$. However, experimentally we find the vanishing $[\deg_{\CC} = k].\Ch^{[-x]}(r,r-s;r-a_1, \dots, r-a_n)\Ch^{[x]}(r,s;a_1, \dots, a_n) = 0$ for $k$ odd.
	}
	\begin{equation}
		\Ch^{[-x]}(1,1-s;0, \dots, 0)
		=
		\left(\Ch^{[x]}(1,s;0, \dots, 0)\right)^{-1} .
	\end{equation}
	It has been proved and employed in \cite{CMS+19}.
\end{rem}

\begin{proof}
	Most equations can be proved exploiting properties of Bernoulli polynomials tuned in the right way. We proceed by grouping similar properties.

	\subsubsection*{Proof of properties \eqref{prop:mult:shifts:s} and \eqref{prop:mult:shifts:ai}, which specialise to \eqref{prop:shift:s} and \eqref{prop:shift:ai}}
	Let us recall a few basic facts from the theory of symmetric functions. Let $p_m$, $\sigma_l$, and $h_l$ be the following three bases of symmetric polynomials: power sums, elementary symmetric, and complete homogeneous. Explicitly, for a set of variables $X = (X_1, \dots, X_N)$ we have:
	\begin{equation}
		p_m(X) =
			\sum_{i=1}^N X_i^m,
		\qquad\qquad 
		\sigma_l(X) =
		\!\!\!\!\!\!\!\! \sum_{1 \leq i_1 < \cdots < i_l \leq N} \!\!\!\!\!\!\!\! X_{i_1} \cdots X_{i_l}, 
		\qquad\qquad
		h_l(X) =
		\!\!\!\!\!\!\!\! \sum_{1 \leq i_1 \leq \cdots \leq i_l \leq N} \!\!\!\!\!\!\!\! X_{i_1} \cdots X_{i_l}. 
	\end{equation}
	The generating series of the $\sigma_l$ and of the $h_l$ read
	\begin{equation}\label{eqn:gen:series:symm:poly}
		\sum_{l \ge 0} \sigma_l(X) u^l = \prod_{i=1}^N (1 + X_i u), 
		\qquad \qquad 
		\sum_{l \ge 0} h_l(X) u^l = \prod_{i=1}^N \frac{1}{(1 - X_i u)},
	\end{equation}
	and are related to the power sums by Newton's identities:
	\begin{equation}\label{eqn:Newton:id}
		\exp\Biggl(\sum_{m \ge 1} \frac{(-1)^{m+1}}{m} p_m u^m \Biggr) = \sum_{l \ge 0} \sigma_l u^l, 
		\qquad \qquad 
		\exp\Biggl(\sum_{m \ge 1} \frac{1}{m} p_m u^m \Biggr) = \sum_{l \ge 0} h_l  u^l.
	\end{equation}
	Let us moreover recall that the Bernoulli polynomials satisfy $B_{m+1}(x + 1) = B_{m+1}(x) + (m+1)x^m$ for any non-negative integer $m$ and any complex variable $x$. For a positive integer $N$, we can iterate this property $N$ times to obtain
	\begin{equation*}
		B_{m+1}(x + N) = B_{m+1}(x) + (m+1) p_m(x, x+1, \dots, x + N - 1).
	\end{equation*}
	We can apply the property above for $x = s/r$ and obtaining
	\begin{equation*}
		\frac{B_{m+1}(\frac{s + Nr}{r})}{m(m+1)}
		=
		\frac{B_{m+1}(\frac{s}{r})}{m(m+1)}
		+
		\frac{1}{m} p_m\bigl(\tfrac{s}{r},\tfrac{s}{r}+1,\dots,\tfrac{s}{r}+N-1\bigr).
	\end{equation*}
	As a consequence, we find that the $\Omega$-classes before pushforward to $\overline{\mathcal{M}}_{g,n}$ (that is, $\Ch_{g,n}'(r,s;a)$ on $\overline{\mathcal{M}}_{g,a}^{r,s}$) satisfy the shifting property
	\begin{equation*}
		\Ch'^{[x]}_{g,n}(r,s + Nr;a)
		=
		\Ch'^{[x]}_{g,n}(r,s;a)
		\cdot
		\exp{\Biggl(
			\sum_{m \ge 1} \frac{(-x)^{m}}{m}
				p_m\bigl(\tfrac{s}{r},\tfrac{s}{r}+1,\dots,\tfrac{s}{r}+N-1\bigr) \kappa_m
		\Biggr)}.
	\end{equation*}
	Notice that we can now apply the pushforward $\epsilon_{\ast}$ on both sides and obtain the statement \eqref{prop:mult:shifts:s}: in fact, writing the above class as a sum over stable graphs $\Gamma$ (\emph{cf.} Section \ref{sec:Omega:classes}), $\epsilon_{\ast}$ simply acts by multiplication of the factor $r^{2g - 1 - h^1(\Gamma)}$ which depends for fixed $g$ only on the first Betti number $h^1(\Gamma)$ of each of the stable graphs produced, which is left unchanged by the decoration of $\psi$- or $\kappa$-classes. In other words, we find
	\begin{equation*}
		\Ch^{[x]}_{g,n}(r,s+Nr;a)
		=
		\Ch^{[x]}_{g,n}(r,s;a)
		\cdot
		\exp{\Biggl(
			\sum_{m \ge 1} \frac{(-x)^{m}}{m}
				p_m\bigl(\tfrac{s}{r},\tfrac{s}{r}+1,\dots,\tfrac{s}{r}+N-1\bigr) \kappa_m
		\Biggr)}
	\end{equation*}
	on $H^{\textup{even}}(\overline{\mathcal{M}}_{g,n})$. This proves property \eqref{prop:mult:shifts:s}, which restricts to property \eqref{prop:shift:s} for $N = 1$. Property \eqref{prop:mult:shifts:ai} is proved similarly, but this time employing instead Equations \eqref{eqn:gen:series:symm:poly} and \eqref{eqn:Newton:id} for elementary symmetric polynomials. Property \eqref{prop:mult:shifts:ai} specialises to property \eqref{prop:shift:ai}.

	\subsubsection*{Proof of property \eqref{prop:0:r:sym}}
	Property \eqref{prop:0:r:sym} is obvious from Equation \eqref{eqn:Chiodo:formula} and the identity for Bernoulli polynomials $B_{m+1}(1) = B_{m+1}(0) = B_{m+1}$.

	\subsubsection*{Proof of property \eqref{prop:pullback}}
	Whenever $s$ lies within the range $0 \leq s < r$, by \cite{LPSZ17} the $\Omega$-classes form a CohFT with flat unit, which can be restated precisely as:
	\begin{equation*}
		\pi^{\ast}\Ch_{g,n}(r,s;a_1,\dots,a_n)
		=
		\Ch_{g,n+1}(r,s;a_1,\dots,a_n,s).
	\end{equation*}
	The case $s=r$ is handled by property \eqref{prop:0:r:sym}. We need to perform the extension of $s$ outside the range $[0,r]$ and show that the statement keeps holding true. For this purpose, we start with $s$ outside the range, and we shift $s$ by adding or subtracting $r$ the required amount of times, controlling the shift process by property \eqref{prop:shift:s}. At this point, we perform the pullback of the correction produced and recognise that it gets perfectly reabsorbed this time by means of property \eqref{prop:shift:ai}:
	\begin{equation*}
		\Ch_{g,n}(r,s;a)
		=
		\begin{cases}
			{\displaystyle
				\Ch_{g,n}(r,\braket{s};a)
				\exp{\Biggl(
					\sum_{m \ge 1} \frac{(-1)^{m}}{m}
						p_m\bigl(\tfrac{\braket{s}}{r},\tfrac{\braket{s}}{r}+1,\dots,\tfrac{s}{r}-1\bigr) \kappa_m
				\Biggr)}
			}
			& \text{if $s \ge r$}, \\[3ex]
			{\displaystyle
				\Ch_{g,n}(r,\braket{s};a)
				\exp{\Biggl(
					- \sum_{m \ge 1} \frac{1}{m}
						p_m\bigl(1-\tfrac{\braket{s}}{r},2-\tfrac{\braket{s}}{r},\dots,-\tfrac{s}{r}\bigr) \kappa_m
				\Biggr)}
			}
			& \text{if $s < 0$}.
		\end{cases}
	\end{equation*}
	Here we wrote $s = r[s]+\braket{s}$ for the Euclidean division of $s$ by $r$. Let us focus on the case $s \ge r$. Pulling-back by the forgetful map, we find
	\begin{equation*}
	\begin{split}
		\pi^{\ast}\Ch_{g,n}&(r,s;a_1,\dots,a_n) = \\
			= &
			\pi^{\ast} \Ch_{g,n}(r,\braket{s};a_1,\dots,a_n)
			\exp{\Biggl(
				\sum_{m \ge 1} \frac{(-1)^{m}}{m}
					p_m\bigl(\tfrac{\braket{s}}{r},\tfrac{\braket{s}}{r}+1,\dots,\tfrac{s}{r}-1\bigr) \pi^{\ast}\kappa_m
			\Biggr)} \\
			= &
			\Ch_{g,n+1}(r,\braket{s};a_1,\dots,a_n,\braket{s})
			\exp{\Biggl(
				\sum_{m \ge 1} \frac{(-1)^{m}}{m}
					p_m\bigl(\tfrac{\braket{s}}{r},\tfrac{\braket{s}}{r}+1,\dots,\tfrac{s}{r}-1\bigr)
					\bigl( \kappa_m - \psi_{n+1}^m \bigr)
			\Biggr)}.
	\end{split}
	\end{equation*}
	Here we used the fact $\pi^{\ast}$ is a ring homomorphism, together with the flat unit property for $\Ch_{g,n}(r,\braket{s})$ and the relation $\pi^{\ast}\kappa_m = \kappa_m - \psi_{n+1}^m$. We can now absorb the exponential of $\kappa$-classes, shifting the $\Omega$-class from $\braket{s}$ back to $s$, then apply Newton's identity~\eqref{eqn:Newton:id} and the generating series for elementary symmetric polynomials~\eqref{eqn:gen:series:symm:poly}:
	\begin{equation*}
	\begin{split}
		\pi^{\ast}\Ch_{g,n}&(r,s;a_1,\dots,a_n) = \\
			= &\,
			\Ch_{g,n+1}(r,s;a_1,\dots,a_n,\braket{s})
			\exp{\Biggl(
				\sum_{m \ge 1} \frac{(-1)^{m+1}}{m}
					p_m\bigl(\tfrac{\braket{s}}{r},\tfrac{\braket{s}}{r}+1,\dots,\tfrac{s}{r}-1\bigr) \psi_{n+1}^m
			\Biggr)} \\
			= &\,
			\Ch_{g,n+1}(r,s;a_1,\dots,a_n,\braket{s})
			\sum_{l \ge 0} \sigma_l\bigl(\tfrac{\braket{s}}{r},\tfrac{\braket{s}}{r}+1,\dots,\tfrac{s}{r}-1\bigr) \psi_{n+1}^l \\
			= &\,
			\Ch_{g,n+1}(r,s;a_1,\dots,a_n,\braket{s})
			\prod_{t = 1}^{[s]} \left( 1 + \left( \frac{s}{r} - t \right) \psi_{n+1} \right)\\
			= &\,
			\Ch_{g,n+1}(r,s;a_1,\dots,a_n,s).
	\end{split}
	\end{equation*}
	In the last equation we used again property \eqref{prop:shift:ai}. This proves the case $s>r$. The case $s < 0$ is obtained by applying Newton's identity for the generating series for complete homogeneous polynomials $h_l$. This concludes the proof of property \eqref{prop:pullback}.

	\subsubsection*{Proof of properties \eqref{prop:string} and \eqref{prop:dilaton}: string and dilaton equations.}
	These are essentially corollaries of property \eqref{prop:pullback} applied to the standard way of proving string and dilaton equations through projection formula.
	For instance, the dilaton equation can be proved as 
	\begin{align*}
		\frac{\partial}{\partial x_{n+1}}
		&
		\int_{\overline{\mathcal{M}}_{g, n+1}}
			\frac{\Omega(r,s; a_1, \dots, a_{n}, s)}{\prod_{i=1}^{n+1} (1 - x_i \psi_i)} \Bigg{|}_{x_{n+1} = 0}
		=
		\int_{\overline{\mathcal{M}}_{g, n+1}}
			\frac{\Omega(r,s; a_1, \dots, a_{n}, s)}{\prod_{i=1}^{n} (1 - x_i \psi_i)} \psi_{n+1} \\
		& =
		\int_{\overline{\mathcal{M}}_{g, n+1}}
			\pi^{\ast}\big[\Omega(r,s; a_1, \dots, a_{n})\big] \prod_{i=1}^n \sum_{d_i \ge 0} x_i^{d_i}\psi_i^{d_i} \cdot \psi_{n+1}\\
		& =
		\int_{\overline{\mathcal{M}}_{g, n+1}}
			\pi^{\ast}\big[\Omega(r,s; a_1, \dots, a_{n})\big] \prod_{i=1}^n \sum_{d_i \ge 0} x_i^{d_i}\left(\pi^{\ast}\psi_i - D_{i,n+1}\right)^{d_i} \cdot \psi_{n+1} \\
		& =
		\int_{\overline{\mathcal{M}}_{g, n+1}}
			\pi^{\ast}\big[\Omega(r,s; a_1, \dots, a_{n})\big] \prod_{i=1}^n \sum_{d_i \ge 0} x_i^{d_i}\left(\pi^{\ast}(\psi_i^{d_i}) - D_{i,n+1}\pi^{\ast}(\psi_i^{d_i-1})\right) \cdot \psi_{n+1} \\
		& =
		\int_{\overline{\mathcal{M}}_{g, n+1}}
			\pi^{\ast}\big[\Omega(r,s; a_1, \dots, a_{n})\big] \prod_{i=1}^n \sum_{d_i \ge 0} x_i^{d_i}\pi^{\ast}(\psi_i^{d_i})  \cdot \psi_{n+1} \\
		& =
		\int_{\overline{\mathcal{M}}_{g, n+1}}
			\pi^{\ast}\Bigg[\Omega(r,s; a_1, \dots, a_{n}) \prod_{i=1}^n \sum_{d_i \ge 0} x_i^{d_i}\psi_i^{d_i} \Bigg]  \cdot \psi_{n+1}
		=
		\int_{\overline{\mathcal{M}}_{g, n}} \frac{\Omega(r,s; a_1, \dots, a_{n})}{\prod_{i=1}^{n} (1 - x_i \psi_i)} \pi_{\ast} \psi_{n+1} \\
		& =
		\int_{\overline{\mathcal{M}}_{g, n}}
			\frac{\Omega(r,s; a_1, \dots, a_{n})}{\prod_{i=1}^{n} (1 - x_i \psi_i)} \kappa_0
		=
		(2g - 2 + n) \int_{\overline{\mathcal{M}}_{g, n}} \frac{\Omega(r,s; a_1, \dots, a_{n})}{\prod_{i=1}^n (1 - x_i \psi_i)}.
	\end{align*}
	Here $D_{i,n+1}$ is the divisor given by the locus of curves with a rational component attached by a single node and containing the two marked points $i$ and $n+1$, satisfying the constraints $D_{i,n+1}D_{j,n+1} = D_{i,n+1}\psi_{n+1} = 0$. Here the convention that negative powers of $\psi$-classes are zero is used. This concludes the proof of property \eqref{prop:dilaton}. The string equation or property \eqref{prop:string} is proved in a similar way. This concludes the proof of the theorem.
\end{proof}

\subsection{Some vanishing of the $\Omega$-integrals}

As an application of the properties above, we provide several vanishing results for integral of $\Omega$-classes with weighted $\psi$-classes. Again, we will write $s = r[s]+\braket{s}$ for the Euclidean division of an integer $s$ by a natural number $r$.

%

\begin{thm}\label{thm:vanishing}
	Fix $g,n \geq 0$ integers such that $2g - 2 + n > 0$. Let $r$ and $s$ be integers with $r$ positive, and $1 \le a_1, \ldots, a_n \le r$ integers satisfying the modular constraint $a_1+a_2+\cdots+a_n \equiv (2g-2+n)s \pmod{r}$. We have
	\begin{equation}
		\int_{\overline{\mathcal{M}}_{g,n+1}}
			\!\!\!\!\!
			\Ch_{g,n+1}^{[x]}(r,s;a_1,\dots,a_n,s) = 0 \qquad \text{ for any } s \in \ZZ.
	\end{equation}
\end{thm}

\begin{proof} 
	The proof is an immediate consequence of property \eqref{prop:pullback}, as the pullback $\pi^{\ast}$ preserves the cohomological degree, whereas the target moduli space is higher in dimension.
\end{proof}
	
By employing Properties \eqref{prop:mult:shifts:s} and \eqref{prop:mult:shifts:ai} above, we obtain the following vanishing.

\begin{cor} ~
	\begin{enumerate}[(i)]
		\item If $s \geq r$ we have:
		\begin{equation}
		\begin{aligned}
			\int_{\overline{\mathcal{M}}_{g,n+1}} &
				\Ch_{g,n+1}^{[x]}(r,s;a_1,\dots,a_n,\braket{s})
				\prod_{t = 1}^{[s]} \left( 1 + \left( \frac{s}{r} - t \right) \psi_{n+1} x \right)
			= 0 \\
			\int_{\overline{\mathcal{M}}_{g,n+1}} &
				\Ch_{g,n+1}^{[x]}(r,\braket{s};a_1,\dots,a_n, s)
				\exp{\biggl(
					\sum_{m \ge 1} \frac{(-1)^{m}}{m}
					p_m\bigl(\tfrac{\braket{s}}{r},\tfrac{\braket{s}}{r}+1,\dots,\tfrac{s}{r}-1\bigr) \kappa_m x^m
				\biggr)}
			= 0
		\end{aligned}
		\end{equation}

		\item If $s < 0$ we have:
		\begin{equation}
		\begin{aligned}
			\int_{\overline{\mathcal{M}}_{g,n+1}} &
				\Ch_{g,n+1}^{[x]}(r,s;a_1,\dots,a_n,\braket{s})
				\prod_{t = 0}^{-[s]- 1} \left( 1 + \left( \frac{s}{r} + t \right) \psi_{n+1} x \right)^{-1}
			= 0 \\
			\int_{\overline{\mathcal{M}}_{g,n+1}} &
				\Ch_{g,n+1}^{[x]}(r,\braket{s};a_1,\dots,a_n,s)
				\exp{\biggl( -
					\sum_{m \ge 1} \frac{(-1)^{m}}{m}
					p_m\bigl(\tfrac{\braket{s}}{r},\tfrac{\braket{s}}{r}+1,\dots,\tfrac{s}{r}-1\bigr) \kappa_m x^m
				\biggr)}
			= 0
		\end{aligned}
		\end{equation}
	\end{enumerate}
\end{cor}

\begin{rem}
 	The statements above can be expressed in terms of the generalised Stirling numbers $\mathsf{s}$ and $\mathsf{S}$ of first and second type respectively (see e.g. \cite{Cha18}):
 	\begin{equation}
 	\begin{aligned}
 		&
 		\int_{\overline{\mathcal{M}}_{g,n+1}}
 			\!\!\!\!\!
 			\Ch_{g,n+1}(r,s;a_1,\dots,a_n,\braket{s})
 			\sum_{m \ge 0} (-1)^{[s]} \mathsf{s}\bigl( [s], [s] - m,\tfrac{\braket{s}}{r} \bigr) \psi_{n+1}^m
 			= 0
 		&\quad \text{for $s > r$,} \\
 		&
 		\int_{\overline{\mathcal{M}}_{g,n+1}}
 			\!\!\!\!\!
 			\Ch_{g,n+1}(r,s;a_1,\dots,a_n,\braket{s})
 			\sum_{m \ge 0} \mathsf{S}\bigl( [s] + m, [s], \tfrac{\braket{s}}{r} \bigr) \psi_{n+1}^m
 			= 0 
 		&\quad \text{for $s < 0$.}
 	\end{aligned}
 	\end{equation}
 	By \cite{Cha18}, we have the following expression in terms of the usual Stirling numbers $\stirlingfirst{a}{b}$ and $\stirlingsecond{a}{b}$ of the first and second type respectively:
 	\begin{equation}
 	\begin{aligned}
 		(-1)^{k} \mathsf{s}\bigl( k, k - m, x \bigr)
 			& = 
 			\sum_{i=0}^m \binom{k + i - m}{i}
 				\stirlingfirst{k}{k - m + i} x^i
 		&\qquad \text{for $s > r$,} \\
 		\mathsf{S}\bigl(k + m, k, x \bigr) 
 			& =
  			\sum_{i=0}^m \binom{m - k - 1}{i} (-1)^{i}
 				\stirlingsecond{m - i - k}{ - k} x^i
 		&\qquad \text{for $s < 0$.}
 	\end{aligned}
	\end{equation}
\end{rem}

\appendix
\section{A technical lemma}

The following lemma appears in some form in the literature, see for instance \cite[Lemma~2.3]{Pix13}. We provide nevertheless a proof here for reader's convenience.

\begin{lem}\label{lem:push:kappa}
	Let $g \geq 0$ and $n > 0$ be integer such that $2g - 2 + n > 0$. Then for any class $\alpha \in H^{\textup{even}}(\overline{\mathcal{M}}_{g,n})$ and any sequence $(u_m)_{m \ge 1}$ of complex numbers, we find
	\begin{equation}
	\begin{split}
		&
		\int_{\overline{\mathcal{M}}_{g,n}}
			\alpha \cdot \prod_{i=1}^{n} \psi_i^{d_i} \exp\Biggr( \sum_{m \ge 1} u_m \kappa_m \Biggl) \\
		&\qquad\qquad\qquad
		=
		\int_{\overline{\mathcal{M}}_{g,n}}
			\alpha \cdot \prod_{i=1}^{n} \psi_i^{d_i}
		+
		\sum_{\ell \ge 1} \frac{1}{\ell!} \sum_{\mu_1,\dots,\mu_\ell \ge 1}
		\int_{\overline{\mathcal{M}}_{g,n + \ell}}
			\pi_{\ell}^{\ast}\alpha
			\cdot
			\prod_{i=1}^{n} \psi_i^{d_i}
			\prod_{j=1}^{\ell} v_{\mu_j} \psi_{n + j}^{\mu_j + 1},
	\end{split}
	\end{equation}
	where the map $\pi_{\ell} \colon \overline{\mathcal{M}}_{g,n + \ell} \to \overline{\mathcal{M}}_{g,n}$ is the morphism forgetting the last $\ell$ marked points, and the sequence $(v_k)_{k \ge 1}$ is obtained from $(u_m)_{m \ge 1}$ from the expansion
	\begin{equation}\label{eqn:change:u:v}
		\exp\Biggl( - \sum_{m \ge 1} u_m x^{m} \Biggr) = 1 - \sum_{k \ge 1} v_{k} x^{k}.
	\end{equation}
\end{lem}

\begin{proof}
	Let us first recall that multi-indexes $\kappa$-classes and single-index $\kappa$-classes are related by 
	\begin{equation*}
		\exp\Biggl( \sum_{m \ge 1} u_m \kappa_m \Biggr)
		=
		1 + \sum_{\ell \ge 1} \frac{1}{\ell!}
			\sum_{\mu_1, \dots,\mu_{\ell} \ge 1} \Biggl( \prod_{j=1}^{\ell} v_{\mu_j} \Biggr) \kappa_{\mu_1,\dots,\mu_{\ell}},
	\end{equation*}
	where $u_m $ and $v_{k}$ are related by the expansion \eqref{eqn:change:u:v}. As by definition $\kappa_{\mu_1,\dots,\mu_{\ell}} = \pi_{\ell,\ast} \bigl(\psi_{n+1}^{\mu_{1} + 1} \cdots \psi_{n+\ell}^{\mu_{\ell} + 1}\bigr)$, the projection formula implies that
	\begin{align*}
		\int_{\overline{\mathcal{M}}_{g,n}} \alpha \cdot \prod_{i=1}^{n} \psi_i^{d_i} \exp\Biggr( \sum_{m \ge 1} u_m \kappa_m \Biggl)
		= &
		\int_{\overline{\mathcal{M}}_{g,n}} \alpha \cdot \prod_{i=1}^{n} \psi_i^{d_i} \\
		& +
		\sum_{\ell \ge 1} \frac{1}{\ell!} \sum_{\mu_1, \dots,\mu_{\ell} \ge 1}
			\int_{\overline{\mathcal{M}}_{g,n + \ell}} \pi_{\ell}^{\ast} \left(\alpha \cdot \prod_{i=1}^{n} \psi_i^{d_i} \right) \prod_{j=1}^{\ell} v_{\mu_j} \psi_{n + j}^{\mu_j + 1}.
	\end{align*}
	Observe that 
	\begin{equation*}
		\pi_{\ell}^{\ast}\left(\alpha \cdot \prod_{i=1}^{n} \psi_i^{d_i} \right)
		=
		\pi_{\ell}^{\ast}\alpha \cdot \prod_{i=1}^{n} \Big( \pi_{\ell}^{\ast}\psi_i \Big)^{d_i} 
	\end{equation*}
	and moreover recall that, if $\pi \colon \overline{\mathcal{M}}_{g,n+1} \to \overline{\mathcal{M}}_{g,n}$ is the forgetful map forgetting the $(n+1)$-th marked point, then $\pi^{\ast}(\psi_i) = \psi_i - D_{i,n+1}$ on $\overline{\mathcal{M}}_{g,n+1}$, where $D_{i,n+1}$ is the Poincaré dual of the divisor represented by the curves with a single node separating a rational component with exactly two leaves marked $i $ and $n+1$ from the other component. As a consequence, $\pi_{\ell}^{\ast}(\psi_i) = \psi_i - \bar{D}_{i,n+1}$, where $\bar{D}_{i,n+1}$ is the Poincaré dual of the divisor represented by the curves with a single node separating a rational component decorated by the leaf $i$ and leaves in a non-empty subset $I \subset \set{n+1,\dots,n+\ell}$ from the other component, summing over all such subsets $I$. The important observation is the following: because we are interested in $\kappa$-classes, all the new leaves $n+1, \dots, n+\ell$ obtained by applying the projection formula are decorated by a $\psi$-class to the power at least one. By a dimension argument, a rational component attached to a single node and decorated by the leaves $i$ and $I$ has dimension $|I| - 1$, whereas the $\psi$-classes decorating that component have total degree at least $|I|$. This proves that the terms involving the classes $\bar{D}_{i,n+1}$ vanish in the integral, therefore proving the statement.
\end{proof}

\printbibliography

\end{document}